\documentclass[psamsfonts, amssymb]{trans2-l}
\numberwithin{equation}{section}
\theoremstyle{plain}
\newtheorem{thm}{Theorem}[section]
\newtheorem{theorem}[thm]{Theorem}
\newtheorem{lemma}[thm]{Lemma}

\theoremstyle{remark}
\newtheorem{remark}[thm]{Remark}

\usepackage{amsmath,amsfonts,amssymb}
\usepackage{graphicx}

\title{On the near periodicity of eigenvalues of Toeplitz matrices}
\author{Michael Levitin}
\address{Department of Mathematics, University of Reading, Whiteknights, PO Box 220, Reading
RG6 6AX, United Kingdom}
\email{M.Levitin@reading.ac.uk}
\author{Alexander V. Sobolev}
\address{Department of Mathematics, University College London, Gower Street, London WC1E 6BT, United
Kingdom}
\email{asobolev@math.ucl.ac.uk}
\author{Daphne Sobolev}
\address{Maths \& Computing Faculty. The Open University in London, 1-11 Hawley Crescent, London NW1
8NP, United Kingdom}
\email{ds8788@tutor.open.ac.uk}

\begin{document}
\maketitle

\section{Introduction}

In order to compute the spectrum of a self-adjoint operator $A$ in a Hilbert space $H$
one can approximate $A$ by a sequence of finite-dimensional operators
$A_n = P_n AP_n$, $n = 1, 2, \dots$, where $P_n$ are orthogonal
projections on finite-dimensional subspaces of $H$ with the property $P_n\to I$ in the
strong sense as $n\to\infty$. However, it is well known that the operators $A_n$ may have eigenvalues which in the limit do not converge to spectral points of $A$. Such eigenvalues are termed \textsl{spurious eigenvalues} and their presence is often described as \textsl{spectral pollution} (see e.g. \cite{Dav}, \cite{Lev}). The spectral pollution usually happens in spectral gaps of the operator $A$.
This phenomenon has been extensively studied both in the abstract
setting (see e.g. \cite{Pokr}, \cite{Desc}), and in various special cases
(see e.g. \cite{Daug}, \cite{Rapp}, \cite{Lewin}). In particular, it was shown in \cite{Pokr}, \cite{Desc}, \cite{Sharg}, and \cite{Lev} that the spectral pollution may occur at any point in a gap of the essential spectrum of $A$.

Perhaps, the simplest example illustrating the spectral pollution, is the
classical Toeplitz matrix. Let $a$ be a real-valued piecewise continuous
function on the interval $[-\pi, \pi)$, and let $P_n$ be the orthogonal
projection in $L^2(-\pi, \pi)$ on the subspace spanned by the
exponentials $(2\pi)^{-1/2}e^{-ikx}$, $k = 1, 2, \dots, n$. We define the
Toeplitz operator with the symbol $a$ as
\[
T_n = T_n[a] = P_n A P_n,
\]
where $A$ is the operator of multiplication by $a$. If the range of the function $a$
is a disconnected set, then the spurious eigenvalues, in the limit $n\to\infty$,
fill in the gaps
separating the components of the range. More precisely, it can be inferred from
\cite{Basor} that an open interval $I$ strictly inside the gap
contains $W \log n + O(1)$ eigenvalues of $T_n$, where $W=W(a, I)$ is an explicitly computable constant. The objective of this note is to study in more detail
the spectrum of the Toeplitz matrix for the piecewise constant symbol of the form
\begin{equation}\label{Eqa_x}
a(x)=\left\{
   \begin{array}{ll}
     0, & x\in[-\pi,L), \\
     1, & x\in[L,\pi),
   \end{array}
 \right.
\end{equation}
where $L\in [0, \pi)$. In \cite{Lev}, it was shown numerically that if
$L p = \pi q$ with some (co-prime) integer numbers $p, q$, then the spurious eigenvalues of the operator $T_n[a]$ are ``nearly periodic" in $n$ with a ``period" $\omega=\omega(p,q)$, see formula \eqref{eq:omega} below.
Specifically, let $\lambda^{(k)}\in (0, 1)$ be an eigenvalue of the matrix $T_k[a]$.
Then for each of the operators $T_{k+\omega l}[a]$, $l = 1, 2, \dots$, there exists an eigenvalue $\lambda^{(k+\omega l)}$ such that the difference $\lambda^{(k+\omega(l+1))} - \lambda^{(k+\omega l)}$ tends to zero as $l\to\infty$.
Although the convergence rate of the latter was not estimated in \cite{Lev}, graphs of the eigenvalues qualitatively showed a rate which is faster than the logarithmic filling rate of the gap. As an example, Fig. \ref{fig:Fig1} (see \cite{Lev}) gives a diagram of the numerically computed eigenvalues of $P_nAP_n$ vs. $n$, with $L=\frac{\pi}{2}$. The
symbols ``o" and ``x" mark the eigenvalues for even and odd $n$ respectively. The diagram clearly suggests periodicity with period $4$.

\begin{figure}[hbt!]
\begin{center}
\fbox{\includegraphics[width=0.9\textwidth]{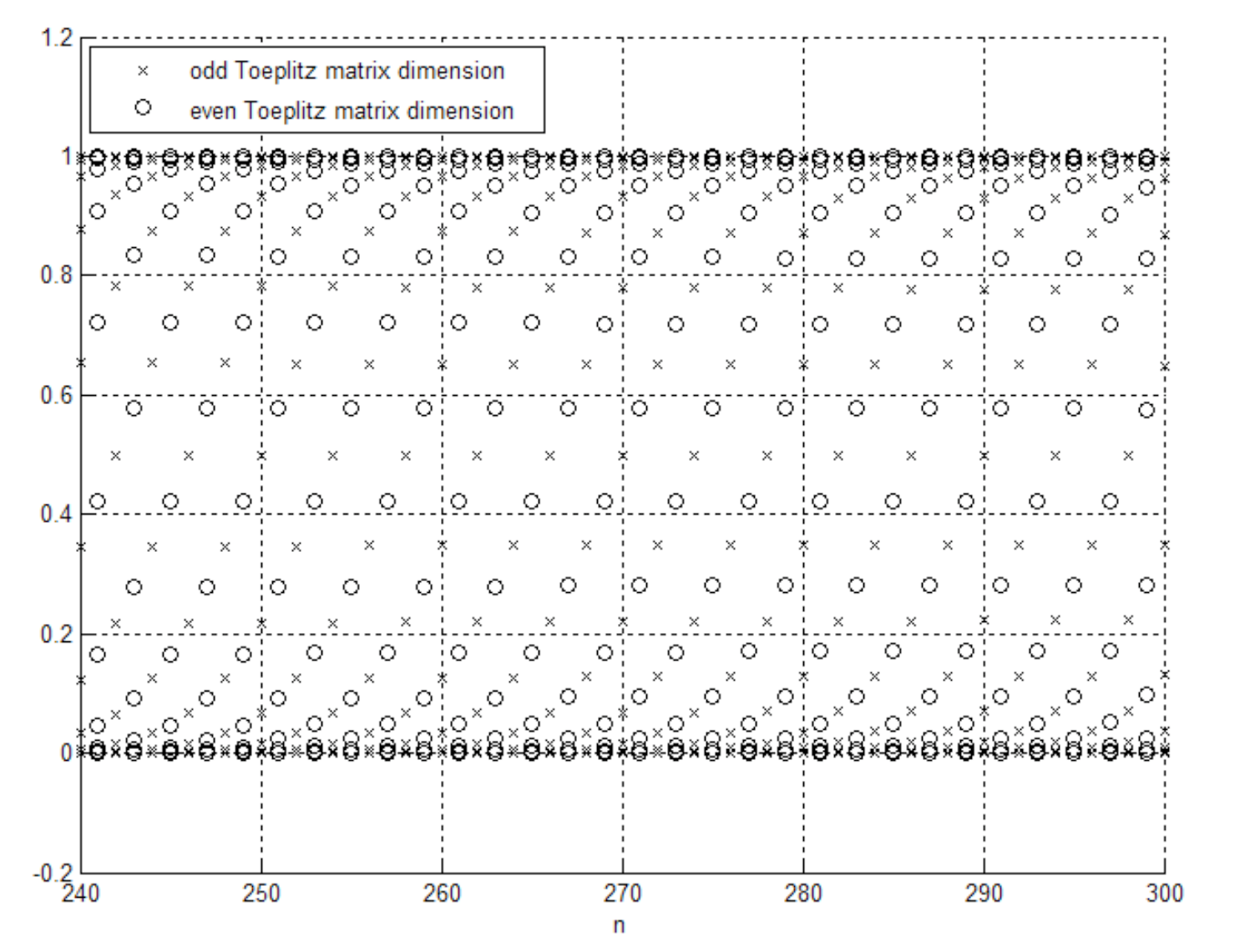}}
\end{center}
\caption{Eigenvalues of $T_n[a]$, with $L=\frac{\pi}{2}$, plotted vs. $n$.}\label{fig:Fig1}
\end{figure}

\noindent
The paper \cite{Lev} (see also \cite{Sharg1} in this volume) offered no quantitative measurement of the properties
of the periodic behavior.
Moreover no rigorous proof has been given, see  \cite{Sharg1}  for some intuitive
discussion of the phenomenon.

In this article we address the question of the near periodicity of eigenvalues, but instead of the symbol \eqref{Eqa_x} it is more convenient to take $a$ of the form
\eqref{a:eq}.
We have not been able to find a proof of this effect for the operator
$T_n[a]$ itself, but we can show its presence for
the squared Toeplitz operator, i.e. for $M_n[a] = (T_n[a])^2$, see Theorem \ref{Th1}.
Numerical examples and more detailed conjectures with regard to the periodicity will be presented in a further publication.

\section{The main result} \label{SectConvergence}

%
We are concerned with the
spectrum of the squared Toeplitz operator
\[
M_n=M_n[a] = (T_n[a])^2.
\]

First, we introduce some consistent notation for eigenvalues of  various operators which appear later in the paper. For any self-adjoint matrix operator $S=S_n$ of size $n$, we shall denote by $\lambda_j(S)$, $j=1,\dots,n$, its eigenvalues labeled in the descending order, and by $\mu_k(S)$, $k=1,\dots,n$,  its eigenvalues labeled in the ascending order,
so that
\[
\mu_j(S)=\lambda_{n+1-j}(S)\,.
\]

Also, for brevity, in the particular case of operators $M_n$, we shall write
\[
\mu_j^{(n)} = \mu_j(M_n)\,,
\]

The next theorem is the main result of the paper:

\begin{theorem} \label{Th1}
Let $a(x)$ be the function
\begin{equation}\label{a:eq}
a(x) =
\begin{cases}
-1,\ x\in[-\pi, L),\\
1,\ x\in [L, \pi).
\end{cases}
\end{equation}
with some $L\in [0, \pi)$.
Suppose that
\begin{equation}\label{EqCondition1}
p L =  \pi q
\end{equation}
for some co-prime   $p\in\mathbb N$ and $q\in \mathbb Z$.
Define
\begin{equation*}\label{eq:omega}
\omega = \omega(p, q)
=
\begin{cases}
2, \ \ q = 0,\\
p, \ \ p\ \  \textup{and}\   q \  \textup{are odd},\\
2p, \ \ \textup{either} \  p\  \textup{or} \ q\  \textup{is even.}
\end{cases}
\end{equation*}
Let $\epsilon \in (0, 1)$ be a fixed number, and let $j$ be such that
$\mu^{(n)}_j< 1-\epsilon$.
Then for a sufficiently large number $K>0$ and $n\ge K \omega  \epsilon^{-1}$,
\begin{equation}\label{mu:eq}
|\mu^{(n)}_j - \mu^{(n+\omega)}_j|\le \frac{C\omega (1+\log^2n)}{\epsilon n},
\end{equation}
with a constant $C>0$ independent of $p, q, n, j, \epsilon$.
\end{theorem}

In general, by $C$ and $c$ we denote various
positive constants independent of $n, p, q$ and $\epsilon$, whose precise value is of no importance.

Note that under the condition \eqref{EqCondition1} we have
\begin{equation}\label{equiv:eq}
\omega (\pi + L) \equiv 0 \mod 2\pi.
\end{equation}

\begin{remark}
Theorem \ref{Th1} immediately implies that
under the condition $\mu^{(n)}_j< 1- \epsilon$, $\epsilon >0$,
for any fixed fixed $N=1, 2, \dots$, and all sufficiently large $n$,
\begin{equation}\label{mus:eq}
|\mu^{(n+\omega(m+1))}_j - \mu^{(n+\omega m)}_j|\le \frac{C \omega (1+\log^2n)}{\epsilon n},
\end{equation}
for all $m= 0, 1, \dots, N$, uniformly in $j$. This means that for
any $N$
the spectra of $M_n[a]$ have strings of length $N$  of ``nearly equal" eigenvalues.
\end{remark}

Throughout the proof of Theorem \ref{Th1}
we omit the subscript $n$ for brevity whenever possible.  Denote
$Q=Q_n = I-P_n$ and compute, remembering that $A^2 = I$:
\begin{equation*}
M_n[a] = (PAP)^2 = PAPAP = PA^2P - PAQAP = P-PAQAP.
\end{equation*}
This means that
\[
\mu^{(n)}_j = 1 - \lambda_j(B_n)\,,
\]
where $B_n = B_n[a] = PAQAP$.
Therefore it suffices to prove the inequality \eqref{mu:eq} for
the eigenvalues $\lambda_j(B_n)$
%
%
which satisfy $\lambda_j(B_n)> \epsilon$,
instead.

The entries of the matrix $B_n$ are easy to find:
\begin{equation}\label{Eqb}
b_{rl} = b^{(n)}_{rl}
= \frac{1}{2\pi} \sum_{\substack{m\le 0,\\ m\ge n+1}} a_{r-m} a_{m-l} ,\ r, l = 1, 2, \dots, n,
\end{equation}
where the Fourier coefficients
\[
a_k = \frac{1}{\sqrt{2\pi}}\int_{-\pi}^{\pi} a(x) e^{ikx} dx,
\]
are given by
\begin{equation}\label{fourier:eq}
a_k=\left\{
  \begin{array}{ll}
    -\sqrt{\dfrac{2}{\pi}} L, & k=0, \\[0.3cm]
    \sqrt{\dfrac{2}{\pi}}\dfrac{(-1)^k}{ik}\left[ 1-e^{ik(L+\pi)} \right], & k\neq 0.
  \end{array}
\right.
\end{equation}
The crucial point of our argument is that due to \eqref{equiv:eq} 
the exponential on the right-hand-side of \eqref{fourier:eq} is $\omega$-periodic 
as a function of $k$. This fact is used only once, in the proof of Lemma \ref{LemmaD}. 

We need to establish the following theorem.

\begin{theorem} \label{Th2}
Let $a(x)$ be as defined in \eqref{a:eq}, and suppose that
the condition \eqref{EqCondition1} is satisfied.
Let $\epsilon \in (0, 1)$ be a fixed number, and let $j$ be such that
$\lambda_j(B_n)> \epsilon$.
Then for all sufficiently large $n\ge 1$,
\begin{equation}\label{mu:eq1}
|\lambda_j(B_n) - \lambda_j(B_{n+\omega})|\le \frac{C\omega (1+\log^2n)}{\epsilon n},
\end{equation}
uniformly in $j$.
\end{theorem}

First we estimate the entries $b^{(n)}_{rl}$.

\begin{lemma} Let $a(x)$ be defined by
\eqref{a:eq}, and let $b^{(n)}_{rl}$ be defined by \eqref{Eqb}.
Then for all $ 1\leq l\le r\le n$ we have:
\begin{equation}\label{upperb:eq}
|b^{(n)}_{rl}|\le \frac{16}{\pi^2}\frac{1+ \log n}{|l-r|},\ l\not = r,
\end{equation}
and
\begin{equation}\label{Eqbl_rEstimate}
|b^{(n)}_{rl}|\le \frac{8}{\pi^2}
\biggl(
\frac{1}{n+1-r} + \frac{1}{l} \biggr).
\end{equation}
\end{lemma}

\begin{proof}
Substituting \eqref{fourier:eq} in \eqref{Eqb} we get
\begin{equation}\label{Eqbl_r}
b_{rl} = -\frac{(-1)^{r-l}}{\pi^2}\sum_{\substack{m\le 0,\\ m\ge n+1}}
\frac{1}{(r-m)(m-l)}\left(1+e^{i(r-l)(L+\pi)}-e^{i(r-m)(L+\pi)}-e^{i(m-l)(L+\pi)}\right).
\end{equation}
Therefore
\begin{align*}
|b_{rl}|\leq &\  \frac{4}{\pi^2}\sum_{\substack{m\le 0,\\ m\ge n+1}}
\left|\frac{1}{(r-m)(m-l)}\right|\\[0.2cm]
= &\ \frac{4}{\pi^2}\biggl(
\frac{1}{|r-(n+1)|\ |(n+1)-l|} + \frac{1}{r l}\biggr)
+ \frac{4}{\pi^2}\sum_{\substack{m\le -1,\\ m\ge n+2}}\frac{1}{|r-m|\ |m-l|}.
\end{align*}
The second term can be easily estimated by an appropriate integral.
If $r>l$, then
\begin{align}
\sum_{\substack{m\le -1,\\ m\ge n+2}}
\frac{1}{|r-m|\ |m-l|}\leq &\
\int_{-\infty}^0 \frac{1}{(r-x)(l-x)}dx+\int_{n+1}^\infty \frac{1}{(x-r)(x-l)}dx \notag\\[0.2cm]
= &\
\frac{1}{(r-l)}\biggl( \log\left(\frac{n+1-l}{n+1-r}
\right)+\log\left(\frac{r}{l}\right)\biggr).\label{ln:eq}
\end{align}
%
%
The bound \eqref{upperb:eq} follows immediately.

Let us derive \eqref{Eqbl_rEstimate}.
In view of the straightforward bounds
\begin{equation*}
\log\left(\frac{n+1-l}{n+1-r}\right)
\leq \frac{r-l}{n+1-r},\qquad \log \bigl(\frac{r}{l}\bigr)\le \frac{r-l}{l},
\end{equation*}
the right hand side of \eqref{ln:eq} does not exceed
\[
\frac{1}{n+1-r} + \frac{1}{l}.
\]
For the case $r=l$, a similar estimate can be obtained:
\begin{align*}
\sum_{\substack{m\le -1,\\ m\ge n+2}} \frac{1}{|r-m|^2}
\leq &\ \int_{n+1}^\infty \frac{1}{(x-r)^2}dx +
\int_{-\infty}^0 \frac{1}{(r-x)^2}dx\\[0.2cm]
= &\ \frac{1}{n+1-r} + \frac{1}{r}.
\end{align*}
Hence for all $r\ge l$ we have:
$$
|b_{rl}|
\leq \frac{8}{\pi^2}
\biggl(
\frac{1}{n+1-r} + \frac{1}{l}
\biggr),
$$
which coincides with \eqref{Eqbl_rEstimate}.
\end{proof}

Let $\omega\in \mathbb N$, $1< \omega<n$, and let
\begin{equation*}
k = k_n=\left\lceil\frac{n}{2}\right\rceil=
\begin{cases}
\dfrac{n}{2}, \quad n  \text{ even}, \\[0.2cm]
\dfrac{n+1}{2}, \quad n  \text{ odd}.
\end{cases}
\end{equation*}
%
%
Let us construct, using matrix $B_{n+\omega}$,  two new auxiliary matrices.
The $(n+\omega)\times(n+\omega)$-matrix $D_{n+\omega}$
has the entries
\begin{equation}\label{Eqc}
d^{(n+\omega)}_{rl}=\left\{
      \begin{array}{ll}
        b^{(n+\omega)}_{rl} & k\leq r< k+\omega,\ \
        \textup{or}\ \ \ k\leq l< k+\omega,\\
        0 & \hbox{otherwise.}
      \end{array}
    \right.
\end{equation}
%
%
The $n\times n$ matrix $F_n$ has the entries
\begin{equation} \label{EqD}
f^{(n)}_{rl}=
\left\{
\begin{array}{ll}
b^{(n+\omega)}_{r,l} & 1\leq r\le k-1, 1\leq l\le k-1,\\[0.2cm]
b^{(n+\omega)}_{r+\omega,l+\omega} & k\leq r\le n, k\leq l \le n, \\[0.2cm]
b^{(n+\omega)}_{r,l+\omega} & 1\leq r\le k-1, k\leq l \le n,\\[0.2cm]
b^{(n+\omega)}_{r+\omega,l} & k\leq r\le n,  1\leq l \le k-1.
\end{array}
\right.
\end{equation}
The method of constructing the matrices $D_{n+\omega}$ and $F_n$
from $B_{n+\omega}$ is illustrated by Figure~\ref{fig:Fig2}. First, we shade the central ``cross" of $\omega$ rows and $\omega$ columns in $B_{n+\omega}$,
starting with row and column number $k$.  The matrix $D_{n+\omega}$ is constructed by replacing everything outside the cross by zeros, and the matrix $F_n$ by ``removing'' the cross, and pulling the remaining four blocks together.

\begin{figure}[hbt!]
\begin{center}
\includegraphics{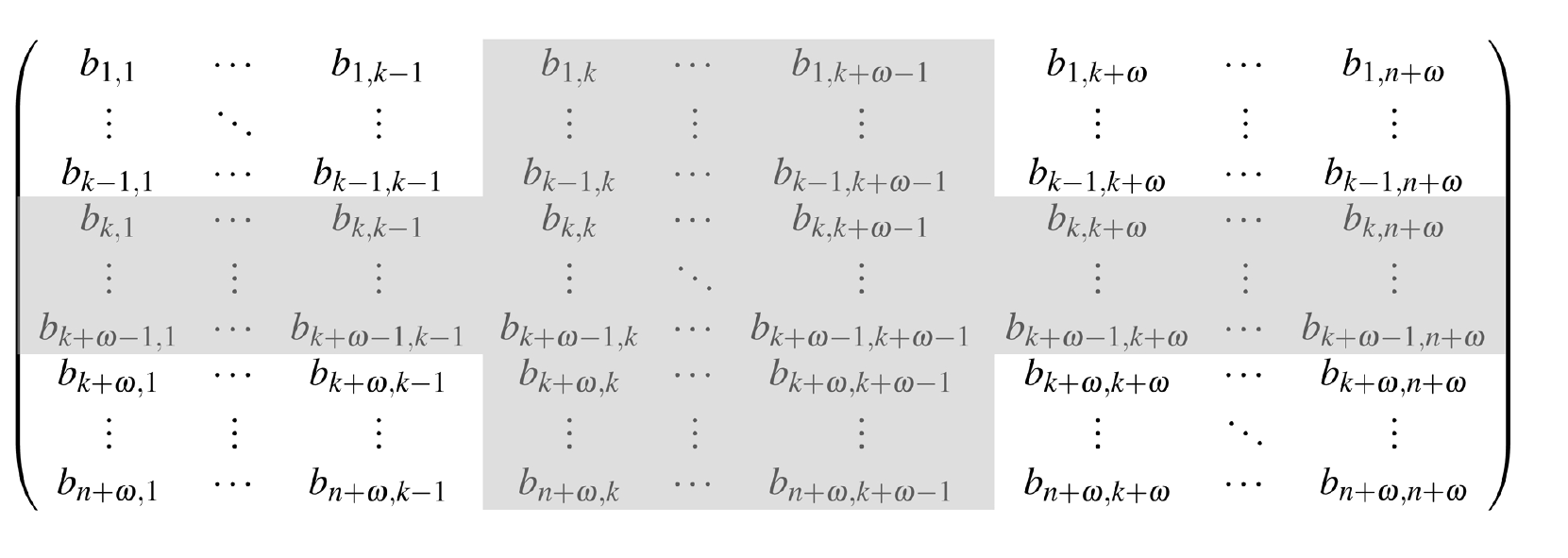}
\end{center}
\caption{Matrix $B_{n+\omega}$ and the central ``cross'' used in the construction of matrices $D_{n+\omega}$ and $F_n$}\label{fig:Fig2}
\end{figure}

Introduce the projections  $\Pi_{n,\omega} = P_{k+\omega-1} - P_{k-1}$, and
$\Xi_{n,\omega} = P_{n+\omega}(I - \Pi_{n,\omega})$.
In terms of these projections, the operator $D^{(n+\omega)}$ can be represented as follows:
\[
D_{n+\omega} = \Pi_{n,\omega} B_{n+\omega} \Pi_{n,\omega}
+ \Pi_{n,\omega} B_{n+\omega} \Xi_{n,\omega}
+ \Xi_{n,\omega} B_{n+\omega} \Pi_{n,\omega},
\]
%
%
%
so that
\[
B_{n+\omega} = \Xi_{n,\omega} B_{n+\omega}\Xi_{n,\omega} + D_{n+\omega}.
\]
The next lemma estimates the matrix $D_{n+p}$ in terms of these projections.

\begin{lemma}\label{LemmaNormC}
Let $b^{(n)}_{rl}$
 be  defined as in
\eqref{Eqb}, and $d^{(n+\omega)}_{rl}$ as in \eqref{Eqc}.
Then for any $\epsilon >0$ and all $n \ge p$ we have
\begin{equation}\label{upperd:eq}
|(D_{n+p} u, u)|\le (\epsilon + C p n^{-1})\| \Pi_{n,p} u\|^2
+ \frac{C p}{\epsilon} \frac{1+\log^2 n}{n}\|\Xi_{n,p} u\|^2,
\end{equation}
for all $u\in L^2(-\pi, \pi)$.
\end{lemma}

\begin{proof}  Write the straightforward estimate:
\begin{align}
|(D_{n+\omega} u, u)| \le &\ |(\Pi_{n,\omega} B_{n+\omega}\Pi_{n,\omega} u,
\Pi_{n,\omega} u)|
+ 2|(\Pi_{n,\omega} B_{n+\omega} \Xi_{n,\omega} u, \Pi_{n,\omega} u)|\notag\\[0.2cm]
\le &\ \|\Pi_{n,\omega} B_{n+\omega}\Pi_{n,\omega}\| \ \|\Pi_{n,\omega} u\|^2
+ 2\|\Pi_{n,\omega} B_{n+\omega} \Xi_{n,\omega} u\| \ \|\Pi_{n,\omega} u\|\notag\\[0.2cm]
\le &\ (\|\Pi_{n,\omega} B_{n+\omega}\Pi_{n,\omega}\| + \epsilon) \|\Pi_{n,\omega} u\|^2
+ \epsilon^{-1}\|\Pi_{n,\omega}B_{n+\omega}\Xi_{n,\omega}\|^2 \|\Xi_{n,\omega} u\|^2.
\label{d:eq}
\end{align}
Here we have used the elementary estimate $2ab\le \epsilon a^2 + \epsilon^{-1} b^2$,
$\epsilon, a, b >0$.

Let us estimate the matrix norms entering the above inequality.
For the first term we estimate the Hilbert-Schmidt
norm using the bound \eqref{Eqbl_rEstimate}:
\begin{equation}\label{pi:eq}
\|\Pi_{n,\omega} B_{n+\omega}\Pi_{n,\omega}\|_{\mathfrak S_2}^2\le
2\frac{16^2}{\pi^2}\sum_{r = k}^{k+\omega-1}
\sum_{l=k}^r \biggl(\frac{1}{n+\omega+1-r} + \frac{1}{l}\biggr)^2
\le C \frac{\omega^2}{k^2}.
\end{equation}
For the matrix
$\|\Pi_{n,\omega}B_{n+\omega}\Xi_{n,\omega}\|$
we also estimate its Hilbert-Schmidt norm, but
now we need
%
%
both \eqref{upperb:eq} and \eqref{Eqbl_rEstimate}:
\begin{equation}\label{split:eq}
\|\Pi_{n,\omega} B_{n+\omega}\Xi_{n,\omega}\|_{\mathfrak S_2}^2\le \biggl(
\sum_{r = k}^{k+\omega-1} \sum_{l=1}^{k-1} + \sum_{r=k+\omega}^{n+\omega} \sum_{l=k}^{k+\omega-1}\biggr)
|b^{(n+\omega)}_{rl}|^2.
\end{equation}
Let us estimate the first sum, which we denote $S_1$. Split it into two parts:
$1\le l \le k/2$ and $k/2 < l\le k-1$. For the first part we use
the estimate \eqref{upperb:eq}, so that
\[
|b^{(n+\omega)}_{rl}|\le C\frac{1+\log n}{n},
\]
and
\begin{equation*}
\sum_{r = k}^{k+\omega-1} \sum_{l\le k/2}
|b^{(n+\omega)}_{rl}|^2\le C\frac{(1+\log n)^2}{n^2} \sum_{r = k}^{k+\omega-1}
\sum_{l=1}^{k/2} 1
\le C \omega \frac{(1+\log n)^2}{n}.
\end{equation*}
The part with $k/2<l\le k-1$ is estimated with the help of \eqref{Eqbl_rEstimate}, so that
\[
|b^{(n+\omega)}_{rl}|\le \frac{C}{n},
\]
and
\begin{equation*}
\sum_{r = k}^{k+\omega-1} \sum_{k/2<l\le k-1}
|b^{(n+\omega)}_{rl}|^2 \le C\frac{\omega}{n}.
\end{equation*}
Thus
\[
S_1 \le C \omega \frac{(1+\log n)^2}{n}.
\]
The same bound holds for the second sum on the right hand side of \eqref{split:eq}.
Together with \eqref{pi:eq} and \eqref{d:eq} these bounds lead to the claimed estimate
\eqref{upperd:eq}.
\end{proof}

\begin{lemma} \label{LemmaD}
Suppose that \eqref{EqCondition1} is satisfied.
Then for all $n\ge 1$ we have
\begin{equation}
\|F_n-B_n\|\leq \frac{C\omega(1+\log n)}{n}.
\end{equation}
\end{lemma}

\begin{proof}
First we estimate the difference $f^{(n)}_{rl} - b^{(n)}_{rl}$.
For convenience we re-write the formula \eqref{Eqb} for $b^{(n+\omega)}_{rl}$:
\begin{align}
b^{(n+\omega)}_{rl} = &\ \frac{1}{2\pi}
\sum_{\substack{m\le 0,\\ m\ge n+\omega+1}} a_{r-m} a_{m-l}\notag\\[0.2cm]
= &\ \frac{1}{2\pi} \sum_{m\le 0}  a_{r-m} a_{m-l}
+ \frac{1}{2\pi} \sum_{m\ge n+1} a_{r-m-\omega} a_{m+\omega-l}.\label{bnp:eq}
\end{align}
Our plan is to estimate the Hilbert-Schmidt norm of $F^{(n)}-B^{(n)}$. So,
we estimate carefully the sum of squares $|f_{rl}  - b_{rl}|^2$ over the four ranges of
$r$ and $l$, specified in \eqref{EqD}.

{\bfseries Case 1:
$1\le r\le k-1$, $1\le l\le k-1$ (upper-left block in Figure~\ref{fig:Fig2}).}
According to \eqref{bnp:eq} we have
\[
f^{(n)}_{rl} - b^{(n)}_{rl} =
\frac{1}{2\pi} \sum_{m\ge n+1} \bigl(a_{r-m-\omega} a_{m+\omega-l}
- a_{r-m} a_{m-l}\bigr).
\]
Using \eqref{fourier:eq} and \eqref{equiv:eq} we get
the bound
\begin{align*}
|a_{r-m-p} &\ a_{m+p-l}
- a_{r-m} a_{m-l}| \\
= &\ \frac{2}{\pi}  |(1-e^{i(r-m)(L+\pi)}) (1-e^{i(m-l)(L+\pi)})|
\biggl|\frac{1}{(r-m-\omega)(m+\omega-l)} - \frac{1}{(r-m)(m-l)}\biggr|\\[0.2cm]
\le &\ \frac{8 \omega}{\pi} \biggl[
\frac{1}{(r-m)^2(m-l)} + \frac{1}{(m-r)(m-l)^2}
\biggr].
\end{align*}
Therefore, we can estimate:
\begin{align*}
|f^{(n)}_{rl} - b^{(n)}_{lr}|
\le &\ \frac{4\omega}{\pi^2} \sum_{m\ge n+1} \biggl[
\frac{1}{(r-m)^2(m-l)} + \frac{1}{(m-r)(m-l)^2}
\biggr]\\[0.2cm]
\le &\
\frac{C\omega}{(n+1 - r)(n+1-l)}\le \frac{C\omega}{n^2},\ 1\le r, l\le k-1,
\end{align*}
and hence
\begin{equation}\label{case1:eq}
\sum_{l, r=1}^{k-1}|f^{(n)}_{rl} - b^{(n)}_{rl}|^2 \le \frac{C\omega^2}{n^2}.
\end{equation}

{\bfseries Case 2: $k\le r \le n$, $k \le l\le n$ (lower-right block in Figure~\ref{fig:Fig2}).}
Using again \eqref{bnp:eq}, from \eqref{EqD} we get:
\[
f^{(n)}_{rl} - b^{(n)}_{rl} =
\frac{1}{2\pi} \sum_{m\le 0} \bigl(a_{r-m+\omega} a_{m-l-\omega}
- a_{r-m} a_{m-l}\bigr).
\]
Arguing as in Case 1, we obtain the same bound:
\begin{equation}\label{case2:eq}
\sum_{l, r=k}^{n}|f^{(n)}_{rl} - b^{(n)}_{rl}|^2 \le \frac{C\omega^2}{n^2}.
\end{equation}

The remaining two cases are trickier:

{\bfseries Case 3: $1\le r \le k-1$, $k \le l\le n$
(upper-right block in Figure~\ref{fig:Fig2}).}
By \eqref{EqD},
\begin{align*}
f^{(n)}_{rl} - b^{(n)}_{lr}
= &\ b^{(n+\omega)}_{r, l+\omega} - b^{(n)}_{lr}\\
= &\ \frac{1}{2\pi}\sum_{m\le 0} \bigl( a_{r-m} a_{m-l-\omega} - a_{r-m} a_{m-l}\bigr) \\
&\ + \frac{1}{2\pi} \sum_{m\ge n+1}\bigl( a_{r-m-\omega} a_{m-l} - a_{r-m} a_{m-l}\bigr).
\end{align*}
Arguing as in the previous case, we obtain
\begin{align*}
| a_{r-m} a_{m-l-\omega} - a_{r-m} a_{m-l}|
\le &\  \frac{8 \omega}{\pi}
\frac{1}{(r-m)(m-l)^2},\ m\le 0,\\[0.2cm]
| a_{r-m-\omega} a_{m-l} - a_{r-m} a_{m-l}|
\le &\ \frac{8 \omega}{\pi}
\frac{1}{(r-m)^2(m-l)},\ m\ge n+1.
\end{align*}
%
%
We estimate
\begin{equation*}
s_{rl}:=\sum_{m\le 0}  \frac{1}{(r-m)(m-l)^2}
\le C\frac{1+\log n}{n^2},
\end{equation*}
and
\begin{equation*}
q_{rl}:=\sum_{m\ge n+1} \frac{1}{(r-m)^2(m-l)}
\le C\frac{1+\log n}{n^2}.
\end{equation*}
Therefore
\begin{equation*}
\sum_{r=1}^{k-1}\sum_{l=k}^n s_{rl}^2
+ \sum_{r=1}^{k-1}\sum_{l=k}^n q_{rl}^2
\le \frac{C(1+\log^2 n)}{n^2}
\end{equation*}
Combining these bounds we arrive at the estimate
\begin{equation}\label{case3:eq}
\sum_{r=1}^{k-1}\sum_{l=k}^n |f^{(n)}_{rl} - b^{(n)}_{rl}|^2
\le \frac{C\omega^2(1+\log^2n)}{n^2}.
\end{equation}

{\bfseries Case 4: $k\le r \le n$, $1 \le l\le k-1$ (lower-left block in Figure~\ref{fig:Fig2}).}
Since the matrices $F^{(n)}$ and $B^{(n)}$ are Hermitian,
we can use the estimate obtained in Case 3, and hence
\begin{equation}\label{case4:eq}
\sum_{r=k}^n\sum_{l=1}^{k-1} |f^{(n)}_{rl} - b^{(n)}_{rl}|^2
\le \frac{C \omega^2(1+\log^2n)}{n^2}.
\end{equation}

{\bfseries End of proof.} We estimate the norm by the Hilbert-Schmidt norm:
\begin{equation*}
\|F_n - B_n\| \le \|F_n - B_n\|_{\mathfrak S_2}.
\end{equation*}
In view of \eqref{case1:eq}, \eqref{case2:eq},\eqref{case3:eq}, \eqref{case4:eq},
\[
\|F_n - B_n\|_{\mathfrak S_2}\le C\omega \frac{1+\log n}{n},
\]
which coincides with the proclaimed estimate.
\end{proof}

Our next step is to estimate the difference between eigenvalues of $B_{n+\omega}$ and
$F_n$. Instead of the eigenvalues themselves, it
is more convenient to work with their counting function. Denote by
$n_+(\lambda, S), \lambda >0,$ the number of eigenvalues
of a Hermitian matrix $S$ strictly above $\lambda$.

\begin{lemma} \label{subspace:lem} Suppose that $\lambda >0$ and
$n\ge K \omega\lambda^{-1}$ with a sufficiently large constant $K>0$. Then
\[
n_+\biggl(\lambda + \frac{C\omega}{\lambda}\frac{1+\log^2 n}{n} , F_n\biggr)
\le n_+(\lambda, B_{n+\omega})\le n_+\biggl(\lambda
- \frac{C\omega}{\lambda}\frac{1+\log^2 n}{n} , F_n\biggr).
\]
\end{lemma}

\begin{proof}
We use the projections $\Xi_{n,\omega}, \Pi_{n,\omega}$ introduced earlier. Recall that
\[
B_{n+\omega} = \Xi_{n,\omega} B_{n+\omega}\Xi_{n,\omega} + D_{n+\omega}.
\]
By Lemma \ref{LemmaNormC} this implies that
\begin{align*}
B_{n+\omega}\le &\ \Xi_{n,\omega} B_{n+\omega}
\Xi_{n,\omega} + \frac{C\omega}{\delta}\frac{1+\log^2 n}{n} \Xi_{n, \omega}
+ (\delta + C\omega n^{-1}) \Pi_{n,\omega},\\[0.2cm]
B_{n+\omega}\ge &\ \Xi_{n,\omega}
B_{n+\omega} \Xi_{n,\omega} - \frac{C\omega}{\delta}\frac{1+\log^2 n}{n} \Xi_{n, \omega}
- (\delta + C\omega n^{-1}) \Pi_{n,\omega},
\end{align*}
for any $\delta >0$.
Thus  the counting function $n_+(\lambda, B_{n+\omega})$
satisfies the following estimates:
\begin{equation}\label{bxi:eq}
n_+(\lambda, B_{n+\omega})\le n_+\biggl(\lambda, \Xi_{n,\omega} B_{n+\omega}
\Xi_{n,\omega} + \frac{C\omega}{\delta}\frac{1+\log^2 n}{n} \Xi_{n,\omega}\biggr)
+ n_+(\lambda, (\delta + C\omega n^{-1}) \Pi_{n,\omega}),
\end{equation}
\begin{equation}\label{bxi1:eq}
n_+(\lambda, B_{n+\omega})\ge n_+\biggl(\lambda, \Xi_{n,\omega} B_{n+\omega}
\Xi_{n,\omega} - \frac{C\omega}{\delta}\frac{1+\log^2 n}{n} \Xi_{n,\omega}\biggr),
\end{equation}
with an arbitrary $\delta >0$. Now take $\delta = \lambda/2$, so that
for sufficiently large $K$ under the condition $n\ge K \omega\lambda^{-1}$ we have
$\delta + C\omega n^{-1} < \lambda$. Therefore the second term on the right-hand-side of
\eqref{bxi:eq} equals zero.
The matrix $F_n$ is obviously similar to $\Xi_{n,\omega} B_{n+\omega}\Xi_{n,\omega}$, so that
their positive eigenvalues coincide. Thus \eqref{bxi:eq} and \eqref{bxi1:eq}
lead to the required inequalities.
\end{proof}

\begin{proof}[Proof of Theorem \ref{Th2}]
By Lemma \ref{LemmaD}, the elementary perturbation theory
yields:
\[
n_+\biggl(\lambda + \frac{C\omega (1+\log n)}{n} , B_n\biggr)
\le n_+(\lambda, F^{(n)})\le
n_+\biggl(\lambda - \frac{C\omega (1+\log n)}{n} , B_n\biggr),
\]
for any $\lambda \in (0, 1]$. Using this bound in combination
with Lemma \ref{subspace:lem}, we get
\begin{equation*}
n_+\biggl(\lambda + \frac{C\omega}{\epsilon}\frac{1+\log^2 n}{n} , B_{n+\omega}\biggr)
\le n_+(\lambda, B_n)\le n_+\biggl(\lambda
- \frac{C\omega}{\epsilon}\frac{1+\log^2 n}{n} , B_{n+\omega}\biggr),
\end{equation*}
for all $\lambda\in [\epsilon, 1]$ and   $n\ge K \omega\epsilon^{-1}$, with a sufficiently large
constant $K$.
This means that if $\lambda_j(B_n)>\epsilon$,
then
\[
|\lambda_j(B_n) - \lambda_j(B_{n+\omega})|\le \frac{C\omega}{\epsilon}\frac{1+\log^2 n}{n},
\]
for all $n\ge K \omega \epsilon^{-1}$.
\end{proof}

As we have already pointed out, Theorem \ref{Th1} follows from Theorem \ref{Th2}
due to the equality $\mu^{(n)}_j = 1 - \lambda_j(B_n)$.
%
%

\end{document}